\documentclass[9pt]{amsart}
\usepackage{amssymb}
\usepackage{color}
\newtheorem{theorem}{Theorem}[section]
\newtheorem{theoremal}{Theorem}[section]

\newtheorem{corollary}[theorem]{Corollary}
\newtheorem{lemma}[theorem]{Lemma}
\newtheorem{proposition}[theorem]{Proposition}
\theoremstyle{definition}
\newtheorem{definition}[theorem]{Definition}
\newtheorem{remark}[theorem]{Remark}

\newtheorem{example}[theorem]{Example}

\numberwithin{equation}{section}

\begin{document}

\title[slice starlike functions over  quaternions]{slice starlike functions over quaternions}

\author[Z.  Xu]{Zhenghua Xu}
\thanks{This work was supported by the NNSF  of China (11371337).}

\address{Zhenghua Xu, School of Mathematics, HeFei University of Technology, Hefei 230601, China}
\email{zhxu$\symbol{64}$hfut.edu.cn}
\author[G.   Ren]{Guangbin Ren}
\address{Guangbin Ren, School of Mathematical Sciences, University of Science and Technology of China, Hefei 230026, China}
\email{rengb$\symbol{64}$ustc.edu.cn}
\begin{abstract}
 In this paper, we initiate the study of  the   geometric function theory for  slice starlike functions over quaternions and its subclasses. This allows us to answer   negatively  some questions about the Bieberbach conjecture, the growth, distortion, and covering theorems for slice regular functions. Precisely, we find that the Bieberbach conjecture holds true for slice starlike functions in contrast to the fact that the Bieberbach conjecture fails for biholomorphic starlike mappings in higher dimensions.
We also establish some   sharp versions of the  growth, distortion, and covering theorems for quaternions.



\end{abstract}
\maketitle
\section{Introduction}
 As a generalization of one complex  variable,  the theory of slice regular functions over quaternions is initiated   by Gentili and Struppa  \cite{GS1,GS2} and further developed  in  \cite{Co,Gentili-Struppa-O,Ghiloni-Perotti}.

 Many results of the  geometric function theory  of one complex variable can be  extended to the setting of slice regular functions over quaternions such as  Schwarz-Pick lemma \cite{Bisi-Stoppato},   Bohr theorem \cite{Rocchetta-Gentili-Sarfatti},   Bloch-Landau theorem \cite{Rocchetta,Xu-Wang}, Landau-Toeplitz theorem \cite{Gentili-Sarfatti}, Borel-Carath\'{e}odory theorem   \cite{Ren},   Julia theory \cite{Ren-Wang-Julia} and invariant metrics for  the quaternionic Hardy space \cite{Arcozzi-Sarfatti}.
However,  some new phenomena  occur  in the quaternionic setting. For example,  it was proven that  no Riemannian metric is invariant under all slice regular self-maps of the quaternionic ball \cite{Arcozzi-Sarfatti}.

However, some sharp  results are unfortunately restricted  to the subclass of  slice regular functions  preserving one slice and some conjectures are believed to hold  true only in this function class. For examples,   Gal,    Gonz\'{a}lez-Cervantes, and  Sabadini in \cite{Gal-Sabadini} suspect  that   the condition of preserving one slice is necessary for the Bieberbach conjecture.
In   \cite{Ren-Wang}, Ren and Wang establish the growth, distortion, and  covering theorem for slice regular functions that preserve one slice via a so-called  convex combination   identity and raised the open question  whether  the condition of preserving one slice is also necessary.

In contrast to the wealthy results of geometric function theory for holomorphic starlike functions,    there is essentially no result in the setting of quaternions except for \cite{Gori-Vlacci}.
One reason  is that without the assumption preserving one slice, all   tools such as the splitting lemma and the convex combination identity fail. The other reason is that the  study of the slice starlike functions  depends heavily on a deep close  relation between the  regular multiplication and  classical point-wise multiplication, which is unfortunately unclear up to now.

The main  aim of this article is to study the  geometric function theory for slice starlike  functions over quaternions. In particular, we shall establish  some coefficient estimates and the theorems concerning growth, distortion, and covering for slice starlike functions over quaternions.

One of the most known coefficient estimates for univalent functions is the celebrated  Bieberbach conjecture    proved by de Branges \cite{Branges} in 1985.

\begin{theoremal}\label{Bieberbach}
Let $f(z)= z+
\Sigma_{n=2}^{+\infty} a_{n}z^{n}$ be  an   injective holomorphic function on the open unit disk $\mathbb{D}=\{z\in \mathbb{C}: |z|<1\}$, then
$$|a_{n}|\leq n, \ \  n = 2, 3, \ldots$$
with strict inequality for all $n$ unless $f$ is a rotation of the Koebe function.
\end{theoremal}

 Cartan  \cite{Cartan} pointed out that this result fails in several complex variables for biholomorphic functions as shown  by providing  a counterexample (cf. \cite{Gong}).  It is quite nature to pose  the
 Bieberbach conjecture in higher dimensions  as follows:

\textit{Let  $\Omega$  be   the open unit ball or polydisk in $\mathbb{C}^{n}(n\geq2)$ and let $f:\Omega \rightarrow \mathbb{C}^{n}$ be a biholomorphic starlike mapping  normalized by $f(0)=0$ and $Df(0)=Id$. Then
$$\frac{\|D^{m}f(0)(z^{m})\|}{m!}\leq m , \quad  m= 2, 3, \ldots, \ \|z\|=1, $$
where $\| \cdot \|$ is the norm of the unit ball  or polydisk.}

Unfortunately,  this  conjecture also fails for the    unit ball. For example, take $$f(z)=(z_{1}+az_{2}^{2},z_{2}), \quad z=(z_{1},z_{2}) \in B^{2},$$
where   $B^{2}=\{ z\in \mathbb{C}^{2} : \| z \| =\sqrt{|z_{1}|^{2}+|z_{2}|^{2}}<1\}$. If $ a= 3\sqrt{3}/2 $, then $f(B^{2})$ is  starlike (see \cite[Example 5]{Roper-Suffridge}). However,  for $z=(0,1)$, we have  $\|D^{2}f(0)(z^{2})\|/2=a>2$.

Recently, a much  weak positive result for the Bieberbach conjecture  has been given  for a subclass of starlike mappings on the unit ball  \cite{Liu-Liu-Xu}. So far there still lacks sharp results in higher dimensions.

Our first main result is to establish a sharp  version  of the de Branges   theorem for    quaternions.  In contrast to the case of    several complex variables, it turns out  that  the Bieberbach   conjecture holds true  for  slice  starlike functions. This seems  to be the first sharp   de Branges   theorem in higher dimensions.



Now we state the definition of  slice starlike functions.
Let $f$ be a   slice regular  function   in the unit ball
$\mathbb{B}$ of quaternions $\mathbb{H}$, normalized by $f(0)=0$ and $f'(0)= 1$.
We further assume that  $\mathcal{Z}_{f} = \{0\}$,
where $\mathcal{Z}_{f}$ denotes  the zero set of   $f$.
Notice that this condition is weaker than $f$ is injective.
Denote this function class  by $\mathcal{S}$, i.e.,
$$\mathcal{S}= \{f: \mathbb{B} \rightarrow \mathbb{H}:  f \ {\rm{is\ slice\ regular    \ such \ that }}\, \ \mathcal{Z}_{f} = \{0\}
\ {\rm{and}}\, \ f '(0) = 1\}.$$
The set of slice starlike functions is     defined as
$$\mathcal{S}^{*}= \big\{f: \mathbb{B} \rightarrow \mathbb{H}:  f \in \mathcal{S}  \ {\rm{ such \ that }}  \ {\rm{Re}}\,  \big( f(q)^{-1} qf'(q)\big) >0  \ {\rm{on}}\, \mathbb{B}   \big\}.$$


In   one complex variable case, a  holomorphic  function  $f$ in $\mathcal{S}^{*}$ is exactly  the starlike function (see Theorem \ref{starlikeness}).  Moreover, it is proven that   the slice regular extension  of any holomorphic starlike function belongs to the function class $\mathcal{S}^{*}$ (see  Example \ref{main-example-2}).  Hence, the  results  obtained for $\mathcal{S}^{*}$  in this paper extend corresponding results for holomorphic starlike functions to  high dimensions.

In fact, the algebraic  condition  in $\mathcal{S}^{*}$ can also be described as a geometric restriction   that the element in  $\mathcal{S}^{*}$ is a slice regular function for which   its modulus is  strictly increasing in the radial. 
Besides, we give some equivalent  descriptions of the slice  starlike function of order $\alpha$   (see Lemma \ref{main-lemma}) which allow us to present  a  new  characterization of holomorphic starlike functions  in   one complex  variable case (see Corollary  \ref{main-corollary}) and  find  that the   so-called slice starlike in \cite[Definition 3.17]{Gal-J-Sabadini}  and  algebraically  starlike in \cite[Definition 5.20]{Gori-Vlacci} are equivalent.

Let us  introduce two important subclasses of slice regular  functions. Denote  by $ \mathbb S$ the unit $2$-sphere of purely imaginary quaternions.  For every $I \in \mathbb S $,    denote by $\mathbb C_I$ the complex  plane $ \mathbb R \oplus I\mathbb R $, isomorphic to $ \mathbb C$. Denote  by $\mathbb{B}_I$ the intersection $ \mathbb{B} \cap \mathbb C_I $.
One subclass of slice regular functions  is given by
$$\mathcal{N} (\mathbb{B}) = \{f: \mathbb{B} \rightarrow \mathbb{H}:  f  \ {\rm{ is \ slice\ regular \ such \ that \ }} f(\mathbb{B}_{I})  \subset  \mathbb{C}_{I}  \ \mathrm{for \ all } \ I \in  \mathbb{S}\}.$$
Another subclass   is given by
$$\mathcal{V} (\mathbb{B}) = \{f: \mathbb{B} \rightarrow \mathbb{H}:  f    \ {\rm{ is \ slice\ regular \ such \ that \ }} f(\mathbb{B}_{I})  \subset  \mathbb{C}_{I} \ \mathrm{for \ some}  \ I \in  \mathbb{S}\}.$$

For the normalized  and injective  function $f \in \mathcal{V} (\mathbb{B})$, a quaternionic version of de Branges theorem    was   established and  a natural question was raised if $\mathcal{V} (\mathbb{B})$ is the largest class of injective  slice regular functions in which the Bieberbach conjecture  holds \cite{Gal-Sabadini}.

In the present paper,    we shall give a negative answer about Bieberbach conjecture in quaternions raised by Gal   et al. in  \cite[p. 1359]{Gal-Sabadini}   and show    that the Bieberbach conjecture  holds for slice starlike functions as follows.
 \begin{theorem}\label{main-2}
 Let $f (q) = q+\sum_{n=2}^{+\infty} q^{n}a_{n} \in \mathcal{S}^{*}$.  Then
$$ |a_{n}|\leq n, \ \   {\rm{for\ all}}\,  n = 2,3,\ldots.$$
   Equality $|a_{n}| = n $ for a given $n  \geq 2$ holds  if and only if
   $$f(q)=q(1-qu)^{-\ast2}, \ \ \ \forall\ q\in \mathbb{B},$$
 for some $u\in \partial \mathbb{B}$.
   \end{theorem}

Note that, by Theorem \ref{main-2} and Example  \ref{main-example-1}, one can find that the Bieberbach conjecture  holds for  many  injective  slice regular functions which are not in $\mathcal{V} (\mathbb{B})$.
Moreover,   the Bieberbach conjecture over quaternions still holds for slice close-to-convex functions (see Theorem \ref{main-slice-convex}).

Related to the Bieberbach conjecture, Fekete and Szeg\"{o} in \cite{Fekete-Szego} proved a  striking inequality
\begin{eqnarray}\label{Fekete-Szego-ineq}
|a_{3}-\lambda a_{2}^{2}|\leq
\left\{
\begin{array}{lll}
3-4\lambda,  &\mathrm {if} \ \lambda\in (-\infty,0),
\\
1+2 e^{-2\lambda/(1-\lambda)},   &\mathrm {if} \ \lambda\in [0,1),
\\
4\lambda-3,  &\mathrm {if} \ \lambda\in [1,+\infty),
\end{array}
\right.
\end{eqnarray}
for any injective holomorphic function  $F(z)= z+\Sigma_{n=2}^{+\infty} a_{n}z^{n}$      on  $\mathbb{D}$.
Moreover,  inequality (\ref{Fekete-Szego-ineq}) is sharp for each  $\lambda\in \mathbb{R}$. For related  complex versions of Theorem \ref{Fekete-Szego-H}, we refer the reader to the literature \cite{Koepf,Keogh-Merkes} and \cite{Duren}.

In this paper, we shall establish the  Fekete-Szeg\"{o} type inequality for slice starlike  functions as follows.
\begin{theorem}\label{Fekete-Szego-H}
Let $f(q) = q+\sum_{n=2}^{+\infty} q^{n}a_{n}\in \mathcal{S}^{*}$.  Then, for any $\lambda\in \mathbb{H}$,
$$|a_{3}-\lambda a_{2}^{2}|\leq \max\{1, |4\lambda-3|\}.$$
 This estimate is sharp for each  $\lambda$. Equality occurs if
 $$f(q)=q(1-qu)^{-\ast2}, \quad   q\in \mathbb{B},$$
 for any $u\in \partial \mathbb{B}$.
  \end{theorem}
Note that if $\lambda=0$ in Theorem \ref{Fekete-Szego-H}, then this result is already obtained in Theorem \ref{main-2}.

As a very important consequence of   Theorem \ref{Bieberbach} for the second-order coefficient, one may deduce  the  following  well-known growth and distortion theorems.
\begin{theoremal}\label{Th:dg-theorem}
Let $F(z)= z+\Sigma_{n=2}^{+\infty} a_{n}z^{n}$
be an injective holomorphic function on
$\mathbb D $. Then for each $z\in\mathbb D$, the following inequalities hold:
\begin{eqnarray}\label{eq:1-Holo}
\frac{|z|}{(1+|z|)^2}\leq |F(z)|\leq \frac{|z|}{(1-|z|)^2};
\end{eqnarray}
\begin{eqnarray}\label{eq:2-Holo}
\frac{1-|z|}{(1+|z|)^3}\leq |F'(z)|\leq \frac{1+|z|}{(1-|z|)^3};
\end{eqnarray}

\begin{eqnarray}\label{eq:3-Holo}
\frac{1-|z|}{1+|z|}\leq \bigg|\frac{zF'(z)}{F(z)}\bigg|\leq \frac{1+|z|}{1-|z|}.
\end{eqnarray}
Moreover, equality holds for one of these six inequalities at some $z_0\in \mathbb D\setminus\{0\}$ if and only if  $$F(z)=\frac{z}{(1-e^{i\theta} z)^2}\ ,\qquad \forall\,\, z\in \mathbb D,$$ for some $ \theta \in \mathbb R$.
\end{theoremal}

However,  the growth and distortion theorems fail   generally in $\mathbb{C}^{n}$  ($n\geq2$)
and only hold for some subclass  such as  starlike or convex mappings
(cf. \cite{Gong-conv,Graham-Kohr}).
Based on  Theorem \ref{Th:dg-theorem},  the growth and distortion
theorems   were formulated  for   normalized injective slice regular functions in the special class $\mathcal{N} (\mathbb{B})$ (see \cite[Theorem 3.11]{Gal-Sabadini}).    See  \cite{Ren-Wang} for further  extensions  to    normalized injective (on $\mathbb{B}_{I}$)  slice regular functions in $\mathcal{V} (\mathbb{B})$.

In this  paper,  we  shall   answer an open problem in   \cite{Ren-Wang} and show  that the  class $\mathcal{V} (\mathbb{B})$  is  not the largest subclass  of slice regular functions such that the  corresponding growth, distortion,  and covering  theorems hold.  Indeed, it is proven   that   the growth, distortion,  and covering  theorems    are  valid in a tighter form for    $\mathcal{S}^{*}$  by applying    a new growth theorem for the Carath\'{e}odory class  in the quaternionic setting.  They are essentially the first significant
theorems belonging to the theory of quaternions itself for that the Taylor coefficients of slice regular functions need not belong to one slice.


Notice that the close analytic relation between convex and starlike functions  known as Alexander's theorem asserts that   $F$ is convex if and only if $zF'(z)$ is starlike for the holomorphic function $F$ on $\mathbb{D}$ normalized by  $F(0)=0$ and $F'(0)=1$. In view of this result,   we give another coefficient estimate  for a subclass of slice regular functions from Theorem \ref{main-2}.

\begin{theorem}\label{main-convex}
Let $f (q) = q+\sum_{n=2}^{+\infty} q^{n}a_{n}  $ be a slice regular function on $\mathbb{B}$ such that $qf '(q)\in \mathcal{S}^{*}$.  Then
$$ |a_{n}|\leq 1, \ \   {\rm{for\ all}}\,  n = 2,3,\ldots.$$
Equality $|a_{n}| = 1 $ for a given $n  \geq 2$ holds  if and only if
   $$f(q)=q(1-qu)^{-\ast }, \ \ \ \forall\  q\in \mathbb{B},$$
for some $u\in \partial \mathbb{B}$.
   \end{theorem}

Theorem  \ref{main-convex} implies directly the following  growth theorem:
$$|f(q)|\leq \frac{|q|}{1-|q|},\quad  \forall \ q \in \mathbb{B}, $$
for any slice regular function $f$ on $\mathbb{B}$   such that $f(0)=0 $ and $qf '(q)\in \mathcal{S}^{*}$.

For normalized   convex functions $F$ on $\mathbb{D}$, there holds a sharper growth theorem
 \begin{eqnarray}\label{Holo-convex-growth}
 \frac{|z|}{1+|z|}\leq  |F(z)|\leq \frac{|z|}{1-|z|},\quad  \forall \ z \in \mathbb{D},
\end{eqnarray}
 which was generalized to $\mathbb{C}^{n}$ (see e.g., \cite{Graham-Kohr,Gong-conv,Liu-Ren}).

Usually,   inequality (\ref{Holo-convex-growth}) is just a consequence of the following distortion theorem in one and several  complex variables.
 \begin{theorem}\label{convex-distortion}  Let $f $ be a slice regular function on $\mathbb{B}$ such that $qf '(q)\in \mathcal{S}^{*}$. Then
$$
\frac{1 }{(1+|q|)^{2}}\leq |f'(q)|\leq \frac{1 }{(1-|q|)^{2}},\quad \forall \ q \in \mathbb{B}.$$
Estimates are sharp. For each $q\in \mathbb{B}\setminus\{0\}$, equality occurs if
 $$f(q)=q(1-qu)^{-\ast}, \quad   q\in \mathbb{B},$$
 for some $u\in \partial \mathbb{B}$.
 \end{theorem}

 However, the method adopted in  one and several
complex variables cannot be used to our setting due to   higher dimensions as well as the non-commutativity   of quaternions and we do not know if the image of $qf'(q)$ is  convex for any $f\in \mathcal{S}^{*}$.  Fortunately, we can obtain a Koebe type  theorem  for slice regular functions with convex   image as follows.
\begin{theorem} \label{covering-for-convex}
Let $f$ be a slice regular function on $\mathbb{B}$ with   convex image and $f'(0)=1$. Then  $f(\mathbb B)$ contains
an  open ball centered at $f(0)$ of radius $ 1/2 $.
Moreover, the constant $ 1/2 $ is optimal.
\end{theorem}

It is worth mentioning  that  many  methods in complex cases   may fail  in the   quaternionic setting since   the  regularity does not keep under usual product and composition  of two slice regular functions due to the  non-commutativity of quaternions. For example,   the subordination argument  in \cite{Graham-Varolin} is  not suitable to prove    Theorem \ref{convex-algebra}. In addition, the non-commutativity of quaternions
 forces us to formulate   new auxiliary functions as in the proof of Theorem \ref{Rogosinski} and to explore new  methods   as in Theorems  \ref{convex-algebra} and \ref{covering-for-convex}.  Especially, Proposition \ref{main-proposition} sets up a bridge
of point-wise and regular multiplication for slice regular functions.

\vskip10pt

The remaining part of this paper is organized as follows.

In Sect. $2$, we set up basic notation  and give some preliminary results from the theory of slice regular functions over quaternions. At  the end of this section, we introduce the concepts   of   slice starlike functions  of order  $\alpha$ and  slice close-to-convex
functions.

In Sect. $3$, we first offer some examples in the space  $\mathcal{S}^{*}$  and establish  some useful lemmas to prove Theorems \ref{main-2}  and \ref{Fekete-Szego-H}.  We also  recall  a Schwarz lemma  from \cite{GS2}  and establish the Rogosinski lemma for slice regular functions which is a sharpened form of the Schwarz lemma.

In Sect. $4$,  we give a coefficient  estimate  for the Carath\'{e}odory class in the quaternionic setting   and  prove  Theorems \ref{main-2} and \ref{Fekete-Szego-H}.

In Sect. $5$,  we shall establish the growth and distortion theorems for some classes of slice regular functions in the quaternionic setting corresponding to Theorem \ref{Th:dg-theorem}. Moreover,  we formulate a result of Hayman  (see Theorem \ref{Hayman-H}) which is a precise  version of the growth theorem.

In Sect. $6$, we shall give some results involving  the radius problems for slice regular functions. Specially,  we establish  the quaternionic analogs   of  Bohr theorem and   Rogosinski theorem  which are of independent interest. As an application of the growth theorem,  the  Koebe  type one-quarter theorem  for the class $\mathcal{S}^{*}$ is also obtained.  Finally, we prove Theorem \ref{covering-for-convex} by Proposition \ref{main-proposition} which  allows   to establish a  Bloch-Landau theorem for slice regular functions with convex image.

\section{Preliminaries}
In this section,  we recall some necessary definitions and preliminary results on slice regular functions.
To have a more complete insight on the theory, we refer the reader to the monograph \cite{Co2,GSS}.

Let $\mathbb H$ denote the non-commutative, associative, real algebra of quaternions with standard basis $\{1,\,i,\,j, \,k\}$,  subject to the multiplication rules
$$i^2=j^2=k^2=ijk=-1.$$
 Every element $q=x_0+x_1i+x_2j+x_3k$ in $\mathbb H$ is composed by the \textit{real} part ${\rm{Re}}\, (q)=x_0$ and the \textit{imaginary} part ${\rm{Im}}\, (q)=x_1i+x_2j+x_3k$. The \textit{conjugate} of $q\in \mathbb H$ is then $\bar{q}={\rm{Re}}\, (q)-{\rm{Im}}\, (q)$ and its \textit{modulus} is defined by $|q|^2=q\overline{q}=|{\rm{Re}}\, (q)|^2+|{\rm{Im}}\, (q)|^2$. We can therefore calculate the multiplicative inverse of each $q\neq0$ as $ q^{-1} =|q|^{-2}\overline{q}$.
 Every $q \in \mathbb H $ can be expressed as $q = x + yI$, where $x, y \in \mathbb R$ and
$$I=\dfrac{{\rm{Im}}\, (q)}{|{\rm{Im}}\, (q)|}$$
 if ${\rm{Im}}\, q\neq 0$, otherwise we take $I$ arbitrarily such that $I^2=-1$,
where  $I $ is an element of the unit 2-sphere of purely imaginary quaternions
$$\mathbb S=\big\{q \in \mathbb H:q^2 =-1\big\}.$$

For every $I \in \mathbb S $, we will denote by $\mathbb C_I$ the plane $ \mathbb R \oplus I\mathbb R $, isomorphic to $ \mathbb C$, and, if $\Omega \subset \mathbb H$, by $\Omega_I$ the intersection $ \Omega \cap \mathbb C_I $.

We can now recall the definition of slice regularity.
\begin{definition} \label{de: regular} Let $\Omega$ be a domain in $\mathbb H$. A function $f :\Omega \rightarrow \mathbb H$ is called (left) \emph{slice} \emph{regular} if, for all $ I \in \mathbb S$, its restriction $f_I$ to $\Omega_I$ is \emph{holomorphic}, i.e., it has continuous partial derivatives and satisfies
$$\bar{\partial}_I f(x+yI):=\frac{1}{2}\left(\frac{\partial}{\partial x}+I\frac{\partial}{\partial y}\right)f_I (x+yI)=0$$
for all $x+yI\in \Omega_I $.
 \end{definition}

\begin{definition} \label{de: derivative}
Let $f :\mathbb{B} \rightarrow \mathbb H$  be a slice  regular function. For each $I\in\mathbb S$, the \textit{$I$-derivative} of $f$ at $q=x+yI$
is defined by
$$\partial_I f(x+yI):=\frac{1}{2}\left(\frac{\partial}{\partial x}-I\frac{\partial}{\partial y}\right)f_I (x+yI)$$
on $\mathbb{B} _I$. The \textit{slice derivative} of $f$ is the function $f'$ defined by $\partial_I f$ on $\mathbb{B}_I$ for all $I\in\mathbb S$.
 \end{definition}

 All slice regular functions on $\mathbb{B}$ can be expressed as power series.
\begin{theorem}\label{Taylor}
Let $f: \mathbb{B} \rightarrow \mathbb H$ be  a  slice regular function. Then
 $$f(q)=\sum\limits_{n=0}^{\infty}q^na_n, \quad  \mathrm{ with} \ a_n=\frac{f^{(n)}(0)}{n!} $$
for all  $q\in \mathbb{B}$.
\end{theorem}

In fact, there are  some  limitations in the   expansion given in Theorem \ref{Taylor}. Hence, a new type of expansion occurs in \cite{Stoppato-2}.

\begin{definition}Let $f$, $g:\mathbb{B} \rightarrow \mathbb H$ be two slice regular functions  of the form $$f(q)=\sum\limits_{n=0}^{\infty}q^na_n,\qquad g(q)=\sum\limits_{n=0}^{\infty}q^nb_n.$$ The regular product (or $\ast$-product) of $f$ and $g$ is the slice regular function defined by$$f\ast g(q)=\sum\limits_{n=0}^{\infty}q^n\bigg(\sum\limits_{k=0}^n a_kb_{n-k}\bigg).$$\end{definition}

 \begin{definition} \label{de: R-conjugate}
Let $ f(q)=\sum\limits_{n=0}^{\infty}q^na_n $ be a slice regular function on $\mathbb{B} $. We define the \emph{regular conjugate} of $f$ as
$$f^c(q)=\sum\limits_{n=0}^{\infty}q^n\bar{a}_n,$$
and the \emph{symmetrization} of $f$ as
$$f^s(q)=f\ast f^c(q)=f^c\ast f(q)=\sum\limits_{n=0}^{\infty}q^n\bigg(\sum\limits_{k=0}^n a_k\bar{a}_{n-k}\bigg).$$
Both $f^c$ and $f^s$ are slice regular functions on $\mathbb{B} $.
\end{definition}

Let $\mathcal{Z}_{f^s}$ denote the zero set of the symmetrization $f^s$ of $f$.
\begin{definition} \label{de: R-Inverse}
Let $f$ be a slice regular function on $\mathbb{B}$. If $f$ does not vanish identically, its \emph{regular reciprocal} is the function defined by
$$f^{-\ast}(q):=f^s(q)^{-1}f^c(q)$$
which  is slice regular on $\mathbb{B}  \setminus \mathcal{Z}_{f^s}$.
\end{definition}

\begin{proposition} \label{prop:Quotient Relation}
Let $f$ and $g$ be slice regular functions on  $\mathbb{B} $. Then for all $q\in \mathbb{B} \setminus \mathcal{Z}_{f^s}$, $$f^{-\ast}\ast g(q)=f(T_f(q))^{-1}g(T_f(q)),$$
where $T_f:\mathbb{B} \setminus \mathcal{Z}_{f^s}\rightarrow \mathbb{B} \setminus \mathcal{Z}_{f^s}$ is defined by
$T_f(q)=f^c(q)^{-1}qf^c(q)$. Furthermore, $T_f$ and $T_{f^c}$ are mutual inverses so that $T_f$ is a diffeomorphism.
\end{proposition}

In one complex variable, the starlike   function  of order $\alpha\in [0,1]$ was first introduced by   Robertson \cite{Robertson}. Now we introduce its  corresponding  version for slice regular functions.
\begin{definition}\label{slice-starlike}
Let $\alpha <1$. We shall say that $f$ is a slice starlike function  of order  $\alpha$ in the unit ball $\mathbb{B}$ if the slice regular function $f\in \mathcal{S}$ satisfies
$${\rm{Re}}\,  \big( f(q)^{-1} qf'(q)\big) >\alpha, \quad q\in \mathbb{B}\setminus \{0\}. $$
Specially, we shall say that $f$ is a slice starlike function for $\alpha=0$.
\end{definition}
Let $\mathcal{S}^{*}({\alpha})$  denote the class of  slice starlike function of order  $\alpha$ and   write $\mathcal{S}^{*}$   instead when  ${\alpha}=0$.

With the concept of slice starlike functions, we can introduce
the concept  of   slice close-to-convex functions given  by
$$\mathcal{C}= \big\{f: \mathbb{B} \rightarrow \mathbb{H}:   \exists \ h\in \mathcal{S}^{*}   \ {\rm{ such \ that }}  \ {\rm{Re}}\,  \big( h(q)^{-1} qf'(q)\big) >0  \ {\rm{on}}\, \mathbb{B}   \big\}. $$

Note that the inclusion $\mathcal{S}^{*} \subset \mathcal{C}$ holds.

In one complex variable case, a  holomorphic function $f$ in
$\mathcal{C}$ is the so-called  close-to-convex function introduced by Kaplan. In \cite[Theorem 1]{Kaplan}, it was proven that  every close-to-convex function is injective which satisfies naturally the Bieberbach conjecture. In fact, this result was proved early  by Reade in  1955 (see e.g., \cite{Kythe,Reade}).

Note that we  have  omitted  the word ``left" for simplicity in the article. Certainly, the right slice regular  functions have completely analogous theories. This observation  shall be used in the proof of Theorem \ref{Rogosinski}.

 \section{Examples and Lemmas}
In this section, we shall offer some examples in the function class $\mathcal{C} $ and give some useful lemmas to  prove Theorems \ref{main-2} and \ref{Fekete-Szego-H}.  In addition, we establish a Rogosinski lemma for slice regular functions.
\begin{example}\label{main-example-1}
Let  $f(q)=q+\sum_{n=2}^{+\infty} q^{n}a_{n}$ be a  slice  regular  function  on $\mathbb{B}$   such  that $ \sum_{n=2}^{+\infty} n|a_{n}|\leq1$. Then    $f$ is injective on $\mathbb{B}$ and is in $\mathcal{S}^{*}$.
\end{example}
\begin{proof}
In fact, for any integer $n\geq2$,
 \begin{equation}\label{inequality-n}
 |q_{1}^{n}-q_{2}^{n}|<n |q_{1}-q_{2}|, \quad \forall \ q_{1}, q_{2}\in \mathbb{B}, q_{1}\neq q_{2}.
\end{equation}
For $n=2$, there holds that
$$q_{1}^{2}-q_{2}^{2}=q_{1}(q_{1}-q_{2}) +(q_{1}-q_{2} )q_{2}.$$
Then, for $ q_{1}, q_{2}\in \mathbb{B}$ with  $q_{1}\neq q_{2}$,
$$|q_{1}^{2}-q_{2}^{2}|\leq |q_{1}(q_{1}-q_{2}) |+|(q_{1}-q_{2} )q_{2}|
<2|q_{1}-q_{2}|.$$
By induction, we can readily obtain (\ref{inequality-n}) for the general case $n\geq2$.

By   the assumption of $f$ and   (\ref{inequality-n}), it follows that, for $ q_{1}, q_{2}\in \mathbb{B}$ with  $q_{1}\neq q_{2}$,
$$|f(q_{1})-f(q_{2})| \geq  |q_{1}-q_{2}|- \sum_{n=2}^{+\infty} | q_{1}^{n}-q_{2}^{n}| |a_{n}| >|q_{1}-q_{2}|\big(1- \sum_{n=2}^{+\infty}n |a_{n}|\big)\geq0,$$
which shows that $f$ is injective on $\mathbb{B}$.

Now let us show  $f\in \mathcal{S}^{*}$.  Under the assumption of $f$, we have
$$ \big|qf'(q)-f(q)\big| \leq \sum_{n=2}^{+\infty} |q|^{n}(n-1)|a_{n}|<|q|- \sum_{n=2}^{+\infty} |q|^{n}|a_{n}| \leq |f(q)| , \quad \forall \ q\in \mathbb{B}\setminus\{0\}.$$
Hence, $\mathcal{Z}_{f} = \{0\}$ and
$$\big|f(q)^{-1} qf'(q)-1\big|=\big|f(q)^{-1}\big|\big|qf'(q)-f(q)\big|<1, \quad \forall \ q\in \mathbb{B}\setminus\{0\},$$
which implies naturally that
${\rm{Re}}\,  \big( f(q)^{-1} qf'(q)\big) >0 $ on $\mathbb{B}$. Thus $f\in \mathcal{S}^{*}$, as desired.
\end{proof}

See   \cite{Goodman} for the complex case of Example  \ref{main-example-1}.
In order to introduce  other important examples, we need to recall   the concept of  starlike   function and a well-known analytical characterization of starlikeness  (see e.g., \cite[p. 36]{Graham-Kohr}).   Let $\Omega\subset \mathbb{R}^{n}$. The set $\Omega$ is called starlike with respect to the origin $0$, if $t\Omega\subset\Omega$ for all $t \in[0,1]$.  
Let $F$ be holomorphic on $\mathbb{D}$. We say that $F$ is a starlike   function  on $\mathbb{D}$ with respect to $0$ if $F$ is injective  on $\mathbb{D}$  and the image $F (\mathbb{D})$ is   starlike  with respect to $0$.
\begin{theorem}\label{starlikeness}
Let $ F: \mathbb{D}\rightarrow \mathbb{C}$ be a holomorphic function with $F(0)= 0$ and $F '(0) = 1$.
Then $F$ is starlike with respect to $0$ if and only if   ${\rm{Re}}\,    \frac{zF'(z)}{F(z)} >0$ on $\mathbb{D}$.
\end{theorem}

We  recall also a so-called convex combination identity for   slice regular functions which was used to establish   sharp growth and distortion theorems for a subclass of slice regular functions (see \cite{Ren-Wang} or \cite{Ren-Wang-Xu}).
\begin{lemma}\label{eq:norm}
Let $f$ be a slice regular function on   $\mathbb{B}$ such that $f(\mathbb{B}_I)\subseteq \mathbb C_I $ for some $I\in \mathbb S$.  Then
for every $ re^{J\theta}\in \mathbb{B}$ with $J\in \mathbb{S}$,
\begin{equation}\label{convex-combination-identity}
\big|f(re^{J\theta})\big|^2 =\frac{1+\langle I,J\rangle}{2}\big|f(re^{I\theta})\big|^2+
\frac{1-\langle I,J\rangle}{2}\big|f(re^{-I\theta})\big|^2,
\end{equation}
where $\langle I,J\rangle=-{\rm{Re}}\,(IJ)\in [-1,1]$.
\end{lemma}

From Theorem \ref{starlikeness}, Lemmas \ref{eq:norm} and \ref{main-lemma} below, one can show that
\begin{example}\label{main-example-2}
Fix $I\in \mathbb{S}$. Let  $F(z)=z+\sum_{n=2}^{+\infty} z^{n}a_{n},  a_{n}\in \mathbb{C}_{I}$ for $n= 2,3,4\ldots$  be a starlike  function  on $\mathbb{B}_{I}\simeq \mathbb{D}$.  Then $f(q)=q+\sum_{n=2}^{+\infty} q^{n}a_{n}\in \mathcal{S}^{*}.$
\end{example}

\begin{example}\label{main-example-3}
 Let   $f \in \mathcal{S}$ be such that
 \begin{eqnarray}\label{condition-C}
 {\rm{Re}}\, \big( \big(f(q)-f(-q)\big)^{-1} qf'(q)\big) >0, \quad \forall \ q\in \mathbb{B} .
\end{eqnarray}
 Then $f\in  \mathcal{C}.$
\end{example}
\begin{proof}
From (\ref{Re-log-m}) below, it follows  that
 $$ r\frac{\partial}{\partial r}\log |f(q)-f(-q)| ={\rm{Re}}\,  \Big( \big(f(q)-f(-q)\big)^{-1} q\big(f'(q)+f'(-q)\big) \Big), $$
for  $q=ru  \in \mathbb{B}\setminus\{0\}$ with $r=|q|$.

Then  condition  (\ref{condition-C}) and Lemma \ref{main-lemma} imply that $(f(q)-f(-q))/2\in \mathcal{S}^{*} $, thus $f\in \mathcal{C}$ by definition.
\end{proof}

To show   Lemma  \ref{main-lemma}, we prove the following useful result  which reveals the close relation between the usual and   regular multiplication.
\begin{proposition}\label{main-proposition}
Let  $f,g $ be two  slice  regular  functions  on $\mathbb{B}$   such  that  $\mathcal{Z}_{f} = \varnothing$. Then we have
\begin{eqnarray}\label{eq:re-inverse}
 {\rm{Re}}\,  \big( f(q)^{-1} g(q)\big) >0 \ {\rm{on}} \,  \  \mathbb{B}
 \Leftrightarrow  {\rm{Re}}\,  \big(  f(q)^{- \ast} \ast   g(q)\big) >0 \  {\rm{on}} \ \mathbb{B},
\end{eqnarray}
 and\begin{eqnarray}\label{modulus-inverse}
|g(q)|<|f(q)|   \ {\rm{on}}\  \mathbb{B}
 \Leftrightarrow  |f(q)^{- \ast} \ast   g(q)| <1 \  {\rm{on}} \ \mathbb{B}.
\end{eqnarray}
  \end{proposition}

\begin{proof}
We only prove the first claim (\ref{eq:re-inverse}). The second can be obtained by the same strategy.
Note that $\mathcal{Z}_{f} = \varnothing$ if and only if   $\mathcal{Z}_{f^c} = \varnothing$ for any  slice  regular  function   on $\mathbb{B}$.   Assume that ${\rm{Re}}\,  \big( f(q)^{-1} g(q)\big) >0$ on $ \mathbb{B}$.
From  the assumption on the functions $f$ and $g$, we have that, by Proposition \ref{prop:Quotient Relation},
\begin{eqnarray}\label{verse-relation}
 f^{-\ast}\ast g(q)=f(T_f(q))^{-1}g(T_f(q)), \quad \forall\ q\in \mathbb{B},
\end{eqnarray}
where  $T_f(q)=f^c(q)^{-1}qf^c(q)\in \mathbb{B}$.
This implies  that  ${\rm{Re}}\,  \big(  f(q)^{- \ast} \ast   g(q)\big) >0 \  $ on $ \mathbb{B}.$

Conversely,  the condition ${\rm{Re}}\,  \big(  f(q)^{- \ast} \ast   g(q)\big) >0 \  $ on $ \mathbb{B}$ implies that ${\rm{Re}}\,  \big(  f(T_f(q))^{-1}g(T_f(q))\big) >0 \  $ on $ \mathbb{B}$ by (\ref{verse-relation}).
Then the desired result  follows from the fact that $T_f\circ T_{f^{c}} (q)=q$ for all $q\in \mathbb{B}$.
\end{proof}

\begin{lemma}\label{main-lemma}
Let  $f $ be a  slice  regular  function  on $\mathbb{B}$   such  that  $\mathcal{Z}_{f} = \{0\}$ and $f '(0) = 1 $. Then the following statements are equivalent:
\begin{enumerate}
\item  ${\rm{Re}}\,  \big( f(q)^{-1} qf'(q)\big) >\alpha $ on $\mathbb{B}\setminus\{0\}$;
\item  ${\rm{Re}}\,  \big(  f(q)^{- \ast} \ast   (qf'(q))\big) >\alpha $ on $\mathbb{B}\setminus\{0\}$;
\item  For any $u\in \partial\mathbb{B}$, the function $M(r)=|f(ru)|/r^{\alpha}$ is strictly increasing on $(0,1)$.
\end{enumerate}
\end{lemma}
\begin{proof}
(1) $ \Leftrightarrow$ (2): It is trivial by Proposition \ref{main-proposition}.

(1)$  \Rightarrow $ (3):
Notice that, for any slice regular function $f(q)$ on $\mathbb{B}$ with $q=ru=|q|u \in \mathbb{B}$, there holds that
$$r\frac{\partial}{\partial r} f(q)= qf'(q).$$
Then, for $q \in \mathbb{B} \setminus\{0\}$,
\begin{eqnarray}\label{Re-log-m}
  r\frac{\partial}{\partial r}\log |f(q)|=\frac{r}{|f(q)|^{2}}{\rm{Re}}\,  \big( \frac{\partial}{\partial r} f(q)\overline{f(q)} \big) ={\rm{Re}}\,  \big( f(q)^{-1} qf'(q)\big)>\alpha.
\end{eqnarray}
Integrating  the above inequality  from  $r_{1}$ to $r_{2}$ with $0<r_{1} <r_{2}<1$ gives that
 $$ \frac{|f(r_{2}u)| }{ |f(r_{1}u)|}>\frac{ r_{2}^{\alpha}  }{  r_{1} ^{\alpha}},$$
which implies  statement (3).

(3)$  \Rightarrow $ (1): Fix  $u\in  \partial\mathbb{B}$.
Under the condition of (3), we have, for $q=ru, r\in (0,1),$
$$ \frac{\partial}{\partial r} \frac{|f(ru)|}{r^{\alpha}} \geq 0,$$
which implies
$$r\frac{\partial}{\partial r}\log |f(ru)| \geq\alpha.$$
This together with (\ref{Re-log-m}) shows that ${\rm{Re}}\,  \big( f(q)^{-1} qf'(q)\big) \geq \alpha $ for all $r\in (0,1)$.
 Since $u\in  \partial\mathbb{B}$ is arbitrary,  it  follows that
 \begin{eqnarray}\label{condition-quality}
  {\rm{Re}}\,  \big( f(q)^{-1} qf'(q)\big)  \geq \alpha, \quad  \forall \ q\in     \mathbb{B}\setminus\{0\}.
\end{eqnarray}
Let us show that inequality (\ref{condition-quality}) is strict.
Consider  the slice regular function
\begin{eqnarray*}
h(q)=
\left\{
\begin{array}{lll}
  f(q) ^{-\ast} \ast (qf'(q))     &\mathrm {if} \ q\in \mathbb{B}\setminus\{0\},
\\
1,   &\mathrm {if} \ q=0.
\end{array}
\right.
\end{eqnarray*}
By  Proposition \ref{main-proposition}, (\ref{condition-quality}) implies that
$${\rm{Re}}\,  h(q)  \geq \alpha, \quad  \forall \ q\in     \mathbb{B}.$$
If equality occurs in (\ref{condition-quality})  for some $q_{0} \in     \mathbb{B}$, then ${\rm{Re}}\,  h(\widetilde{q_{0}})=\alpha $ for     $\widetilde{q_{0} }=f(q_{0})^{-1} q_{0} f(q_{0}) \in     \mathbb{B}.$
By applying  the maximum principle (see \cite[Theorem 7.13]{GSS}) for the real part
of a slice regular function to $-h$, we see that   $h$ is constant on $\mathbb{B}$.
Whence   $f(q) =qf'(q) $ on $\mathbb{B}$, which forces that  $f$ is an identity on $\mathbb{B}$. Thus, for $\alpha<1$, (\ref{condition-quality}) is strict.

Hence  statements (1) and (3) are equivalent.
\end{proof}

Note that $\mathcal{S}^{*}({\alpha})=\emptyset$ if $\alpha \geq 1$ from Lemma \ref{main-lemma}.

As a byproduct,   by Theorem \ref{starlikeness} and Lemma \ref{main-lemma},  we obtain a   characterization for starlike functions in one complex variable case.
\begin{corollary}\label{main-corollary}
Let  $F$ be   holomorphic   on $\mathbb{D}$   such  that  $F(0)=0$ and $F'(0) = 1$. Then the following statements are equivalent.
\begin{enumerate}
\item   $F$ is starlike   with respect to $0$;
\item  ${\rm{Re}}\,  \big(  F(z)^{-1} zF'(z))\big) >0 $ on $\mathbb{D}$;
\item  for any $\theta \in [0,2\pi) $, the function $M(r)=|F(re^{i\theta})|$ is strictly increasing on $[0,1)$.
\end{enumerate}
\end{corollary}
\begin{remark}
 Note that    condition $(2)$ in Corollary  \ref{main-corollary} would imply  that $F$ is injective on $\mathbb{D}$.  From  the   formula
$${\rm{Re}}\,  \big(  \frac{ zF'(z)}{F(z)}\big)=\frac{\partial}{\partial \theta} \arg F(re^{i\theta}),
\quad z= re^{i\theta},$$
it follows that the argument of $F(re^{i\theta})$ increases monotonically for $\theta$. Thus   $F$ is injective on $|z|=r<1$.  By  the principle of univalence on the boundary, we can conclude that $F$ is injective on $|z| \leq r$. Since $r<1$   is arbitrary, $F$ is injective on $\mathbb{D}$  (see e.g., \cite[p. 38]{Graham-Kohr}). However, this approach is not suitable for the quaternionic setting.

\end{remark}

 Now we  state without proof  a variant of the   Schwarz lemma for slice regular functions in \cite[Theorem 4.1]{GS2} which is useful in the sequel.
\begin{theorem}\label{generalized-Schwarz-lemma}
If $f:\mathbb{B}\rightarrow \mathbb{B}$ be   slice  regular   such
that  $f(0)=f'(0)= \ldots=f^{(m-1)}(0)=0 \quad (m\geq1)$.  Then
$$|f(q)|\leq |q|^{m},\qquad\forall\ q\in \mathbb{B},$$
and
$$|f^{(m)}(0)|\leq m!.$$
Both inequalities are strict (except $q=0$) unless $f(q)=q^{m}u$ for some $u\in \partial\mathbb{B}$.
\end{theorem}

In fact, we  can establish a sharpened form  of the  Schwarz lemma for   slice regular functions.  Its complex version is due to Rogosinski (see e.g., \cite[p. 200]{Duren}).

For $q_{0}\in \mathbb{H}$,  denote $B(q_{0},r)=\{q\in \mathbb{H}: |q-q_{0}|<r\}$.
\begin{theorem}\label{Rogosinski}
 Let $q_{0},b \in \mathbb{B}$. For the set of all slice regular functions
$f: \mathbb{B}\rightarrow  \mathbb{B} $ with $f(0)=0$ and $f'(0)=b$, the range of values of $f(q_{0})$ is the
closed ball  $\overline{B(c,r)}$, where
$$c=\frac{q_{0}b(1-|q_{0}|^{2})}{1-|q_{0}b|^{2}}, \quad r=\frac{|q_{0}|^{2}(1-|b|^{2})}{1-|q_{0}b|^{2}}.$$
\end{theorem}
\begin{proof}For $b=0,$ the  assertion is trivial. Now we consider the case for $b\neq0.$
Denote the slice regular function $g(q)=q^{-1}f(q) $ on $\mathbb{B}$ with  $g(0)=b$ and $|g(q)| \leq 1$ for all $q\in \mathbb{B}$ by the   maximum modulus principle for slice regular functions (cf. \cite[Theorem 7.1]{GSS}).
Denote the slice regular function
$$h(q)=(1-g(q)\overline{b})^{-*}*(b-g(q)),$$
 with $h(0)=0$ and  $|h(q)| < 1$ for all $q\in \mathbb{B}$ by  Proposition \ref{prop:Quotient Relation}.
From   Theorem \ref{generalized-Schwarz-lemma}, it follows that
$$|h(q)|\leq |q|, \quad \forall \ q\in \mathbb{B}.$$
That is, by Proposition \ref{prop:Quotient Relation},
$$|(1-g\circ T_{(1-g\overline{b})^{c}}(q)\overline{b})^{-1}(b-g\circ T_{(1-g\overline{b})^{c}}(q))|\leq |q|, \quad \forall \ q\in \mathbb{B}.$$
Since $T_{(1-g\overline{b})^{c}}\circ T_{1-g\overline{b}}(q)=q$ and $|T_{1-g\overline{b}}(q)|=|q|,$ we obtain
$$|(1-g(q)\overline{b})^{-1}(b-g(q))|\leq |q|, \quad \forall\ q\in \mathbb{B}.$$
Equivalently,
$$\frac{|f(q)-qb|}{|f(q)-q\overline{b}^{-1}|}\leq |bq|, \quad \forall\ q\in \mathbb{B}.$$
A direct computation   shows that
 $$\Big|f(q)- \frac{q b(1-|q |^{2})}{1-|q b|^{2}}\Big| \leq \frac{|q |^{2}(1-|b|^{2})}{1-|q b|^{2}}, \quad \forall\ q\in \mathbb{B} .$$
Hence $f(q_{0})\in \overline{B(c,r)}$ with $c$ and $r$ as desired.

To see that $\overline{B(c,r)}$  is covered, let  $p\in \overline{\mathbb{B}}$ and set the slice regular function
$$f(q)=q(1-q|b|p)^{-*}*(|b|-qp)\frac{b}{|b|}, \quad \forall\ q\in \mathbb{B}.$$
 It is  easy to see that   $f(0)=0$ and $f'(0)=b$.
 Now we show that  $|f(q)| < 1$ for all $q\in \mathbb{B}$. To this end, we only prove that
 \begin{eqnarray} \label{dd}
 \big|(1-q|b|p)^{-*}*(|b|-qp)\big|<1,  \quad \forall\ q\in \mathbb{B}.
\end{eqnarray}
Note that
$$ (1-|b|^{2})(1-|qp|^{2})>0,  \quad \forall\ q\in \mathbb{B},$$
that is to say
$$\big|  |b|-qp \big|<\big|1-q|b|p \big|,  \quad \forall\ q\in \mathbb{B}.$$
Therefore, (\ref{dd}) holds   by Proposition \ref{main-proposition}.

 After a lengthy but direct  calculation, there holds that
$$f(q)=\frac{qb(1-|q|^{2})}{1-|qb|^{2}}+ \frac{q^{2}(1-|b|^{2})}{1-|qb|^{2}}\phi(p)\frac{b}{|b|},$$
where
 $$\phi(p)
=\overline{q}|b|-(1-|qb|^{2})\sum_{n=1}^{+\infty}(|b|q)^{n-1}p^{n}
=(|b|\overline{q}-p)*_{r}(1-q|b|p)^{-*_{r}}$$
 is a right slice regular M\"{o}bious transform of $\mathbb{B}$.

Since  $p\in \overline{\mathbb{B}}$  is arbitrary,  it follows that  $ \phi(\overline{\mathbb{B}})=\overline{\mathbb{B}} $.

Note   that, for any $u\in  \partial \mathbb{B} $,
 $$\{qu:  q\in \overline{\mathbb{B}}\}=\overline{\mathbb{B}}.$$
 Hence the closed ball  $\overline{B(c,r)}$ is covered. The proof is complete.
\end{proof}
From Theorem \ref{Rogosinski} and its proof   we obtain a sharper    Schwarz lemma.
\begin{corollary}\label{Rogosinski-corollary}
Let  $f$ be a slice regular function on $\mathbb{B}$ with $f(0)=0$ and $|f(q)|<1$ for all $q\in \mathbb{B}$. Then
$$|q|\frac{|f'(0)|-|q|}{1-|qf'(0)|}\leq |f(q)|\leq|q|\frac{|q|+|f'(0)|}{1+|qf'(0)|}, \quad \forall\ q\in \mathbb{B}.$$
Equality holds   at some point $q_{0} \neq 0$ if
and only if $f(q)=q\varphi_{a}(q)u$   for some $a\in \mathbb{B}, u\in \partial \mathbb{B}$.
\end{corollary}

As an important generalization of the Schwarz  lemma,  the classical Schwarz-Pick lemma states that if $F$ is a holomorphic self-mapping on   $\mathbb{D}$,    then
\begin{equation}\label{classical-Schwarz-Pick}
|F'(z)|\leq \frac{1-|F(z)|^{2}}{1-|z|^{2}}, \qquad \forall\ z\in \mathbb{D}.
\end{equation}
 However, this classical version fails even for the slice regular automorphisms. Indeed, we take   the slice  regular  M\"{o}bius transformation (see \cite[Theorem 7.6]{Stoppato})
$$\varphi_{a}(q)=(1-q\overline{a})^{-*}*(a-q)=a-(1-|a|^{2})\sum_{n=1}^{+\infty}q^{n}
\overline{a}^{n-1}.$$
Pick $a=\frac{i}{2}$ and  $q_0=\frac{j}{2}$.
By direct computation, we have
$$\varphi_{a}(q_0)=\frac{2}{5}(i-j), \qquad \varphi'_{a}(q_0)=-\frac{204}{225}-\frac{96}{225}k,$$
so that
$$|\varphi'_{a}(q_0)|=\sqrt{\frac{50832}{50625}}> \frac{1-|\varphi_{a}(q_0)|^{2}}{1-|q_0|^{2}}=\frac{68}{75} .$$

Fortunately, a variant  of the Schwarz-Pick lemma for self-mapping of $\mathbb{B}$ was established in \cite{Bisi-Stoppato} . As a  special case, one can obtain a    coefficient  estimate.
\begin{lemma}\label{Schwarz-Pick-Cor}
Let  $f $ be a  slice  regular  function  on $\mathbb{B}$   such  that $|f(q)|\leq1$ for all $q\in \mathbb{B}$. Then
 \begin{eqnarray} \label{Schwarz-lemma-12}
|f'(0)|\leq 1-|f (0)|^{2} .
\end{eqnarray}
 Equality holds  in $(\ref{Schwarz-lemma-12})$ if and only if $f$ is of the form
$$f(q)= (1-q\overline{a})^{-*}*(a-q)u,\qquad \forall \  q\in \mathbb B,$$
for some  $a\in \overline{\mathbb{B}}, u \in \partial\mathbb{B}$.
\end{lemma}

\section{Coefficient  Estimates}
As in  the complex holomorphic case, let $\mathcal{P}$  denote the class of slice regular  functions $p$ on $\mathbb{B}$ such that $p(0) = 1$
and ${\rm{Re}}\,   p(q) > 0$ on $\mathbb{B}$. The function class  $\mathcal{P}$ is usually called the Carath\'{e}odory class.

To prove   Theorems  \ref{main-2} and \ref{Fekete-Szego-H}, we shall first  present a coefficient  estimate for the class  $\mathcal{P}$ based on   Lemma  \ref{Schwarz-Pick-Cor}.

\begin{theorem}\label{sharper-Caratheodory-theorem}
Let  $f(q)=1+\sum\limits_{n=1}^\infty q^n a_n$ be  a   function in $\mathcal{P}$. Then
\begin{eqnarray} \label{sharper-C}
\Big|a_{2}-\frac{a_{1}^{2}}{2}\Big|\leq 2- \frac{|a_{1}|^{2}}{2}.
\end{eqnarray}
 Equality holds  in $(\ref{sharper-C})$ if and only if $$f(q)=(q\varphi_{a}(q)u+1)\ast (1-q\varphi_{a}(q)u )^{-\ast},$$
for some  $a\in \overline{\mathbb{B}}$ and $ u \in \partial\mathbb{B}$.
\end{theorem}
\begin{proof}
Define\begin{eqnarray*}
g(q)=
\left\{
\begin{array}{lll}
q^{-1}(f(q)+1)^{-\ast} \ast( f(q)-1),   &\mathrm {if} \ q\in \mathbb{B}\setminus\{0\},
\\
 a_{1}/2 ,   &\mathrm {if} \ q=0.
\end{array}
\right.
\end{eqnarray*}

From the assumption of $f$, we have
$$|(f(q)+1)^{-1} ( f(q)-1)|<1, \quad \forall \ q \in \mathbb{B},$$
then, by Proposition \ref{main-proposition},
the function $h(q)=(f(q)+1)^{-\ast} \ast( f(q)-1)$ is slice regular on $\mathbb{B}$ with $h(0)=0, h'(0)=a_{1}/2$ and $|h(q)|<1$ for all $q\in \mathbb{B}$. By the   maximum modulus principle for slice regular functions (cf. \cite[Theorem 7.1]{GSS}),  we obtain  that the function $g $ is a slice regular function on    $\mathbb B$ with $g'(0)=a_{2}/2-a^{2}_{1}/4$ and $|g(q)|\leq1$ for all $q\in \mathbb{B}$. From  Lemma \ref{Schwarz-Pick-Cor},   (\ref{sharper-C}) follows and equality holds  in $(\ref{sharper-C})$ if and only if  $g(q)=\varphi_{a}(q)u$
 for some $a\in \overline{\mathbb{B}}$ and $ u \in \partial\mathbb{B}$.
 By simple calculations, we deduce that
 $$f(q)=(q\varphi_{a}(q)u+1)\ast (1-q\varphi_{a}(q)u )^{-\ast}.$$
 Now the proof  is  complete.
\end{proof}

 As a direct consequence, we obtain all coefficient  estimates for the Carath\'{e}odory   class $\mathcal{P}$ applying the approach of finite average as in the case of one complex variable. Now  let us recall this result    in \cite{Ren}.
\begin{theorem}\label{Ren}
Let  $f(q)=1+\sum\limits_{n=1}^\infty q^n a_n$ be a slice regular function in $\mathcal{P}$. Then
\begin{eqnarray}\label{eq:11}
\frac{1-|q|}{1+|q|}\leq  {\rm{Re}}f(q)\leq |f(q)|\leq \frac{1+|q|}{1-|q|}, \qquad\forall\ q\in\mathbb B,
\end{eqnarray}
 and \begin{eqnarray}\label{eq:11an}|a_n|\leq 2,\qquad \ \  n = 1, 2, \ldots.  \end{eqnarray}
Moreover, $|a_1|=2$ or equality holds for the first or third inequality in $(\ref{eq:11})$ at some  $q_0\neq 0$ if and only if
$$f(q)=(1-q u)^{-\ast}\ast(1+qu),\qquad \forall\,\, q\in \mathbb B,$$
for some  $u \in \partial\mathbb{B}$.
\end{theorem}

Now we can prove  our main coefficient estimates.
\begin{proof}[Proof of Theorem \ref{main-2}]
Define
\begin{eqnarray*}
p(q)=
\left\{
\begin{array}{lll}
(q^{-1}f(q))^{-\ast} \ast f'(q)     &\mathrm {if} \ q\in \mathbb{B}\setminus\{0\},
\\
1,   &\mathrm {if} \ q=0,
\end{array}
\right.
\end{eqnarray*}
which belongs to $\mathcal{P}$ by Lemma \ref{main-lemma}. Set $p(z) = 1+ qp_{1}+ \cdots$.
Applying    Theorem \ref{Ren} to function $p$,  we have $|p_{n}|\leq 2$ for all $  n  \geq1$. Comparing the coefficients in the power series
of $f'(q) $ and $q^{-1}f(q)\ast p(q)$, we have
 \begin{equation}\label{coefficient}
 na_n =p_{n-1}+a_{2}p_{n-2}+ \cdots +a_{n-1}p_{1}+a_{n},  \ \  n = 2,3, \ldots.
\end{equation}
By induction, we can obtain $ |a_{n}|\leq n $ for all $n = 2,3,\ldots$.

 It is easy to see that if $ |a_{n}|=n$ for a given $n$, then the above arguments
imply that $|p_{1}|=|a_{2}| = 2$, and thus, by Theorem \ref{Ren}, $p$ is of the form
$$p(q)=(1-qu)^{-\ast}\ast(1+qu)=1+ \sum_{n=1}^{+\infty} q^{n}2u^{n}, \ \ \ \forall \ q\in \mathbb{B}$$
 for some $u\in \partial \mathbb{B}$.
Hence, by (\ref{coefficient}), we have
$$(n-1)a_n =2(u^{n-1}+a_{2}u^{n-2} + \cdots +a_{n-1}u),  \ \  n = 2,3, \ldots,$$
which leads to that
$a_n=nu^{n-1}$ for all $n = 2,3, \ldots$, as desired.
\end{proof}
Note that   Theorem \ref{main-2} generalizes Theorem 2.2.16 in  \cite{Graham-Kohr} to the  non-commutative algebra.  Nevanlinna established the complex version of   Theorem \ref{main-2} for normalized starlike functions in  1920.  In fact, the Bieberbach conjecture over quaternions still holds for slice close-to-convex
functions.
\begin{theorem}\label{main-slice-convex}
Let $f (q) = q+\sum_{n=2}^{+\infty} q^{n}a_{n}\in \mathcal{C}$.  Then
$$ |a_{n}|\leq n, \ \   {\rm{for\ all}}\,  n = 2,3,\ldots.$$
   Equality $|a_{n}| = n $ for a given $n  \geq 2$ holds  if and only if
   $$f(q)=q(1-qu)^{-\ast2}, \ \ \ \forall\ q\in \mathbb{B},$$
 for some $u\in \partial \mathbb{B}$.
   \end{theorem}
\begin{proof}
By definition, there exists a   function $h=q+q^{2}h_{2}+ \cdots\in \mathcal{S}^{*}$   such that
 \begin{equation}\label{def-function}
 p(q) = (q^{-1}h(q))^{-\ast} \ast f'(q)=1+ qp_{1}+ q^{2}p_{2}+\cdots,\ \ \  q\in \mathbb{B},
\end{equation}
which belongs to $\mathcal{P}$ by Lemma \ref{main-lemma}.

From   (\ref{def-function}), we have
$$f'(q)=q^{-1}h(q)\ast p(q)  ,\ \ \  q\in \mathbb{B}.$$
Comparing the coefficients in the power series above,  we have
 \begin{equation}\label{coefficient-2}
 na_n =p_{n-1}+h_{2}p_{n-2}+ \cdots +h_{n-1}p_{1}+h_{n},  \ \  n = 2,3, \ldots.
\end{equation}

Applying    Theorem \ref{Ren} to function $p$,  we obtain that  $|p_{n}|\leq 2$ for all $  n  \geq1$.  From the assumption of $h$,  by Theorem \ref{main-2}, we have  $ |h_{n}|\leq n $ for all $n = 2,3,\ldots$. Combining these  two estimates with (\ref{coefficient-2}), it follows that
$ |a_{n}|\leq n $ for all $n = 2,3,\ldots$.

Following the proof in Theorem \ref{main-2}, the  condition of equality  can be obtained. This completes the proof.
\end{proof}

\begin{proof}[Proof of Theorem \ref{Fekete-Szego-H}]
Let us consider the function
\begin{eqnarray*}
g(q)=
\left\{
\begin{array}{lll}
(qf'(q)+ f (q))^{-\ast} \ast \big(qf'(q)-f (q) \big)     &\mathrm {if} \ q\in \mathbb{B}\setminus\{0\},
\\
0,   &\mathrm {if} \ q=0.
\end{array}
\right.
\end{eqnarray*}
Then it easy to see that $g$ is slice regular on $\mathbb{B}$ with
$$  g_{1}=g'(0)=\frac{a_{2}}{2} ,\ \ g_{2}=\frac{g''(0)}{2}= a_{3}-\frac{3}{4}a_{2}^{2}.$$
From the assumption of $f$, we have
$$\big|(qf'(q)+ f (q))^{-1}   \big(qf'(q)-f (q) \big)\big|<1, \quad \forall \ q\in \mathbb{B}.$$
By Lemma \ref{main-lemma}, $|g(q)|<1$ for all $q\in \mathbb{B}$, which gives that,  by Lemma \ref{Schwarz-Pick-Cor},
$$ |g_{2}-\lambda g_{1}^{2}|\leq 1-|g_{1}|^{2} + |\lambda||g_{1}|^{2}  \leq \max\{1, |\lambda|\},$$
that is
$$\Big|a_{3}-\frac{3}{4}a^{2}_{2}-\lambda \frac{1}{4}a^{2}_{2}\Big| \leq \max\{1, |\lambda|\}.$$
Hence,
$$|a_{3}-\lambda a_{2}^{2}|\leq \max\{1, |4\lambda-3|\}.$$
The sharpness of the estimate  can be easily
checked from the function given in the theorem.
Now the proof is complete.
\end{proof}

\section{Growth and Distortion Theorems}
In this section, we shall study  the growth  and distortion theorems for   slice regular functions.

From Corollary  \ref{Rogosinski-corollary},  we first present a growth theorem for  the Carath\'{e}odory   class $\mathcal{P}$ which is an improved  version of Theorem \ref{Ren}.
\begin{theorem} \label{sharper-Ren}
Let  $p(q)=1+\sum\limits_{n=1}^\infty q^n a_n$ be a slice regular function in $\mathcal{P}$. Then
\begin{eqnarray}\label{sharper-Ren-inequa}
\frac{1-|q|^{2}}{1+|a_1 q|+|q|^{2}}\leq  {\rm{Re}}p(q)\leq |p(q)|\leq \frac{1+|a_1 q|+|q|^{2}}{1-|q|^{2}}, \qquad\forall\ q\in\mathbb B.
\end{eqnarray}
Moreover,   equality holds for the first or third inequality in $(\ref{sharper-Ren-inequa})$ at some  $q_0\neq 0$ if and only if
$$p(q)=(q\varphi_{a}(q)u+1)\ast (1-q\varphi_{a}(q)u )^{-\ast},\qquad \forall\,\, q\in \mathbb B,$$
for some  $u \in \partial\mathbb{B}$.
\end{theorem}
\begin{proof}
Consider the slice regular function
 $$f(q)= (p(q)+1)^{-\ast}\ast(p(q)-1), \qquad\forall\ q\in\mathbb B.$$

Note that  $f(0)=0, f'(0)=a_{1}/2$ and $|f(q)|<1$ for all $q\in \mathbb{B}$ by Proposition \ref{main-proposition}.
then  it follows that, by Corollary  \ref{Rogosinski-corollary},
\begin{equation}\label{equality-hold}
|f(q)|\leq|q|\frac{|q|+|f'(0)|}{1+|qf'(0)|}=:\mu, \quad \forall\ q\in \mathbb{B}.
\end{equation}
By Proposition  \ref{prop:Quotient Relation}, it follows that
$$ \frac{|p(q)+1|}{|p(q)-1|} \leq \mu, \quad \forall\ q\in \mathbb{B}. $$
i.e.,
 $$\Big|p(q)- \frac{1+\mu^{2}}{1-\mu^{2}}\Big| \leq \frac{2\mu }{1-\mu^{2}}, \quad \forall\ q\in \mathbb{B}.$$
Hence, (\ref{sharper-Ren-inequa}) follows from
$$ |p(q)|\leq \frac{1+\mu^{2}}{1-\mu^{2}}  + \frac{2\mu }{1-\mu^{2}}  =\frac{1+|qa_1|+|q|^{2}}{1-|q|^{2}},  \quad \forall\ q\in \mathbb{B}.$$
and
$$ {\rm{Re}}p(q)\geq\frac{1+\mu^{2}}{1-\mu^{2}} - \frac{2\mu }{1-\mu^{2}}  =\frac{1-|q|^{2}}{1+|qa_1|+|q|^{2}}, \quad \forall\ q\in \mathbb{B}.  $$
If one of equalities  above  occurs  for  some point $q_0\neq 0$, then (\ref{equality-hold}) turns into an equality at $\widetilde{q_0}=T_{p^{c}+1}(q_{0})=(p(q_{0})+1)^{-1} q_{0}(p(q_{0})+1)$.
Whence  Corollary  \ref{Rogosinski-corollary} implies  that
  $$f(q)=q\varphi_{a}(q) u,\quad \forall \,\,q\in \mathbb B,$$
for some $u\in \partial \mathbb{B}$. Thus
$$p(q)=(q\varphi_{a}(q)u+1)\ast (1-q\varphi_{a}(q)u )^{-\ast},\quad \forall \ q\in \mathbb B, $$
for some $u\in \partial \mathbb{B}$.

The converse part can be verified easily by   direct calculations.
Now the proof is complete.
\end{proof}

Now we can establish  the growth and distortion  theorems for $\mathcal{S}^{*}$.
\begin{theorem}\label{growth-distortion-sharp}
Let $f(q)=q+\sum\limits_{n=2}^\infty q^n a_n\in \mathcal{S}^{*}$.    Then the following inequalities hold for all $q \in \mathbb{B}$
 \begin{eqnarray}\label{eq:1-sharp}
 \frac{|q|}{ 1+|a_2q|+|q|^{2} } \leq|f(q)|\leq \frac{|q|}{ 1-|q| ^{2}} \Big(\frac{1+|q|}{ 1-|q| }\Big)^{ \frac{|a_{2}|}{2}};
\end{eqnarray}
\begin{eqnarray}\label{eq:2-sharp}
 \frac{1-|q|^{2}}{(1+|a_2 q|+|q|^{2})^{2}} \leq|f'(q)|\leq \frac{1+|a_2 q|+|q|^{2}}{(1-|q|^{2})^{2}}\Big(\frac{1+|q|}{ 1-|q| }\Big)^{ \frac{|a_{2}|}{2}};
\end{eqnarray}
\begin{equation}\label{eq:1-2-sharp}
 \frac{1-|q|^{2}}{1+|a_2 q|+|q|^{2}}  \leq      \frac{|qf'(q)|}{|f(q)|}\leq
\frac{1+|a_2 q|+|q|^{2}}{1-|q|^{2}}.
\end{equation}
\end{theorem}

\begin{proof}Following the proof of  Theorem \ref{main-2}
 and applying Theorem \ref{sharper-Ren}, we can obtain a combined growth and distortion theorem for    $f\in \mathcal{S}^{*}$ and all $q\in \mathbb{B}$
$$ \frac{1-|q|^{2}}{1+|a_2 q|+|q|^{2}} \leq  {\rm{Re}} \Big( f(q)^{-*}*(qf'(q))\Big) \leq \Big| f(q)^{-*}*(qf'(q))\Big| \leq \frac{1+|a_2 q|+|q|^{2}}{1-|q|^{2}}.$$
Thus, by Proposition \ref{prop:Quotient Relation},
\begin{equation}\label{growth-Hayman-1-sharp}
  \frac{1-|q|^{2}}{1+|a_2 q|+|q|^{2}}\leq  {\rm{Re}} \Big( f(q)^{-1}qf'(q)\Big) \leq  \frac{|qf'(q)|}{|f(q)|}\leq \frac{1+|a_2 q|+|q|^{2}}{1-|q|^{2}},
\end{equation}
which implies   two inequalities in (\ref{eq:1-2-sharp}).

From (\ref{Re-log-m})  and (\ref{growth-Hayman-1-sharp}), we have
 $$ \frac{1-r^{2}}{r(1+|a_2|r+r^{2})} \leq\frac{\partial}{\partial r}\log |f(ru)|  \leq \frac{1+|a_2|r+r^{2}}{r(1-r^{2})},$$
 where $q=ru$ with $r=|q|$.

Integrating  the above inequality  along a
radius from  $r_{1}$ to $r_{2}$ with $0<r_{1} <r_{2}<1$ gives that
\begin{equation}\label{incre-Hayman}
  \log  \frac{r}{1+|a_2|r+r^{2}} \Big|_{r_{1}  }^{r_{2}}  \leq \log |f(ru)| \Big|_{r_{1}  }^{r_{2}}
  \leq \log \frac{r}{(1+r)^{1-\frac{|a_{2}|}{2}}(1-r)^{1+\frac{|a_{2}|}{2}}}
  \Big|_{r_{1}  }^{r_{2}}.
\end{equation}
Equivalently,
$$ \frac{r_{2} }{1+|a_2|r_{2}+r_{2}^{2}} \varepsilon_{1}(r_1) \leq    |f(r_{2}u)| \leq     \frac{ r_{2}   }{ 1- r_{2} ^{2}} \Big(\frac{1+ r_{2} }{ 1- r_{2} }\Big)^{ \frac{|a_{2}|}{2}}
  \varepsilon_{2}(r_1), $$
where
$$\varepsilon_{1}(r_1)=(1+|a_2|r_{1}+r_{1}^{2}) \frac{|f(r_{1}u)| }{r_{1}},  $$
and
$$\varepsilon_{2}(r_1)=  (1- r_{1} ^{2}) \Big(\frac{1- r_{1} }{ 1+r_{1} }\Big)^{ \frac{|a_{2}|}{2}}
\frac{|f(r_{1}u)|}{r_{1}}.  $$
Letting $r_{1}\rightarrow 0^{+}$, (\ref{eq:1-sharp}) follows from
$$ \frac{r_{2} }{1+|a_2|r_{2}+r_{2}^{2}}   \leq    |f(r_{2}u)| \leq     \frac{ r_{2}   }{ 1- r_{2} ^{2}} \Big(\frac{1+ r_{2} }{ 1- r_{2} }\Big)^{ \frac{|a_{2}|}{2}} . $$

Now (\ref{eq:2-sharp}) is a consequence of (\ref{eq:1-sharp}) and (\ref{eq:1-2-sharp}).    The proof is complete.
\end{proof}

Note that $|a_{2}|\leq2$ provided  $f(q)=q+\sum\limits_{n=2}^\infty q^n a_n\in \mathcal{S}^{*}$. Hence we immediately obtain the following growth and distortion theorem for $\mathcal{S}^{*}$.
\begin{theorem}\label{growth-distortion}
Let $f\in \mathcal{S}^{*}$.    Then the following inequalities hold for all $q \in \mathbb{B}$
 \begin{eqnarray}\label{eq:1}
 \frac{|q|}{(1+|q|)^{2}} \leq|f(q)|\leq \frac{|q|}{(1-|q|)^{2}};
\end{eqnarray}
\begin{eqnarray}\label{eq:2}
 \frac{1-|q|}{(1+|q|)^{3}}\leq|f'(q)|\leq \frac{1+|q|}{(1-|q|)^{3}};
\end{eqnarray}
\begin{equation}\label{eq:1-2}
  \frac{1-|q|}{1+|q|}\leq      \frac{|qf'(q)|}{|f(q)|}\leq \frac{1+|q|}{1-|q|}.
\end{equation}
All of these estimates are sharp. For each $q\in \mathbb{B}\setminus\{0\}$, equality occurs if
 $$f(q)=q(1-qu)^{-\ast2}, \quad   q\in \mathbb{B},$$
 for some $u\in \partial \mathbb{B}$.
   \end{theorem}

As  a consequence of (\ref{eq:1}), Theorem \ref{convex-distortion} follows.    Now we establish a more general version.
\begin{theorem}\label{convex-algebra}  Let  $m $ be a positive integer  and $f(q) = q+\sum_{n=m+1}^{+\infty} q^{n}a_{n} $ be a slice regular function on $\mathbb{B}$ such that $qf '(q)\in \mathcal{S}^{*}$. Then
$$\frac{1 }{(1+|q|^{m})^{2/m}}\leq |f'(q)|\leq \frac{1 }{(1-|q|^{m})^{2/m}},\quad \forall \ q \in \mathbb{B}.$$
 \end{theorem}
 \begin{proof}
Denote the slice regular function
$$g(q)=f'(q)^{-\ast} \ast( qf''(q)), \quad \forall \ q \in \mathbb{B}.
$$
From the assumption of $f$, we have
$$ {\rm{Re}} \Big( f'(q) ^{-1} qf''(q) \Big) >-1, \quad \forall \ q \in \mathbb{B},$$
then, by Proposition \ref{main-proposition}, $g$ is slice regular on $\mathbb{B}$ with ${\rm{Re}}g(q)>-1$ for all $q\in \mathbb{B}$.
Then
  \begin{eqnarray}\label{B-B}
 |(2+g(q))^{-1}g(q)|<1, \quad \forall \ q \in \mathbb{B}.
\end{eqnarray}

Consider  the slice regular function
$$h(q)=(2 +g(q))^{-\ast} \ast g(q) = (2f'(q)+qf''(q))^{-\ast} \ast( qf''(q)), \quad \forall \ q \in \mathbb{B}, $$
with $h(0)=h'(0)=\ldots =h^{(m-1)}(0)=0$ and $|h(q)|<1$  on $ \mathbb{B}$  by  Proposition \ref{main-proposition} again and (\ref{B-B}).
Whence the Schwarz lemma for slice regular functions in Theorem \ref{generalized-Schwarz-lemma} implies that $|h(q)|\leq|q|^{m}$ for all $q \in \mathbb{B}$. Then
$$|h\circ T_{2+g^{c}}(q)|\leq|T_{2+g^{c}}(q)|^{m}=|q |^{m}, \quad  \forall  \ q \in \mathbb{B}.$$
 That is, by  Proposition  \ref{prop:Quotient Relation},
$$ |(2+g(q))^{-1}g(q)| \leq |q |^{m}, \quad \forall \ q \in \mathbb{B}, $$
equivalently,
$$ \Big| g(q) -\frac{2r^{2m}}{1-r^{2m}}\Big| \leq \frac{2r^{ m}}{1-r^{2m}}, \quad |q|=r<1,$$
which implies that
$$ \Big| g\circ T_{(f')^{c}}(q) -\frac{2r^{2m}}{1-r^{2m}}\Big| \leq \frac{2r^{ m}}{1-r^{2m}}, \quad |q|=r<1,$$
i.e., by  Proposition  \ref{prop:Quotient Relation},
$$ \Big| f'(q) ^{-1} qf''(q)-\frac{2r^{2m}}{1-r^{2m}}\Big| \leq \frac{2r^{ m}}{1-r^{2m}}, \quad |q|=r<1.$$
Specially,
 $$ \Big|{\rm{Re}} \Big( f'(q) ^{-1} qf''(q)\Big)-\frac{2r^{2m}}{1-r^{2m}}\Big| \leq \frac{2r^{ m}}{1-r^{2m}}, \quad |q|=r<1.$$
This together with (\ref{Re-log-m}) yields that
 $$ -\frac{2r^{m-1 }}{1+r^{ m}} \leq \frac{\partial}{\partial r}\log | f' (ru)| \leq \frac{2r^{m-1}}{1-r^{ m}}, \quad  q =ru.$$
 Integrating  the above inequality   from  $0$ to $r$  gives
   $$ -\frac{2}{m}\log(1+r^{ m}) \leq \log | f' (ru)| \leq -\frac{2}{m}\log(1-r^{ m}), \quad  q =ru,$$
as  desired.
\end{proof}

 From Theorem  \ref{convex-algebra}, we deduce a precise growth theorem  for $\mathcal{S}^{*}$.
 \begin{theorem}\label{starlike-algebra-higher}
 Let  $f(q) = q+\sum_{n=m+1}^{+\infty} q^{n}a_{n}  \in \mathcal{S}^{*}$ for some positive integer $m $. Then$$\frac{|q|}{(1+|q|^{m})^{2/m}}\leq |f (q)|\leq \frac{|q| }{(1-|q|^{m})^{2/m}},\quad \forall \ q \in \mathbb{B}.$$
 \end{theorem}

In fact, estimates on  the right side of (\ref{eq:1}) and (\ref{eq:2}) can be directly obtained from   coefficient estimates in  Theorem  \ref{main-2}.  Equality on the right side of   (\ref{eq:1}) or (\ref{eq:2}) occurs for some $q_{0}\in \mathbb{B} \setminus\{0\}$, if and only if
 $$f(q)=q(1-qu)^{-\ast2}, \quad   q\in \mathbb{B},$$
 for some $u\in \partial \mathbb{B}$.

Moreover, those two  estimates hold in a   larger class. Let $f, g$ be two slice regular functions on $\mathbb{B}$. Write $f \prec_{\mathcal{N}} g$ if there exists $w\in \mathcal{N}(\mathbb{B} )$ with $w(0) = 0$ and $|w(q)|<1$  for all $q \in  \mathbb{B}$ such that $f (q) =  g ( w(q))$ for all $q\in \mathbb{B}$ (see   \cite[Definition 2.10]{Gal-J-Sabadini}). Note that the composition of  slice regular functions, when defined, does not give    generally a slice regular function. However, if
$ w \in\mathcal{N} (\mathbb{B}),$  then the composition $f\circ  w$ is slice regular for any slice regular function $f$.
Since  the proof of  Theorem \ref{growth-distortion-suborinate} below  is similar to the classical case (see e.g., \cite[p. 202]{Duren}), we omit it here.
\begin{theorem}\label{growth-distortion-suborinate}
Let $f, g$ be two slice regular functions on $\mathbb{B}$.  If  $f  \in \mathcal{C}$ and $g \prec_{\mathcal{N}} f$, then the following inequalities hold for all $q \in \mathbb{B}$
 $$|g(q)|\leq \frac{|q|}{(1-|q|)^{2}}\  \  \mathrm {and } \  \
 |g'(q)|\leq \frac{1+|q|}{(1-|q|)^{3}}.$$
\end{theorem}

 In fact, we can obtain more information related to (\ref{eq:1}). Note that  the author  in \cite{Hayman} established the Hayman's regularity theorem  making use of complex version of Theorem \ref{Hayman-H} for the $p$-univalent holomorphic function.
\begin{theorem}\label{Hayman-H}
Let $f \in\mathcal{S}^{*}$ and $M_{\infty}(r,f)=\max_{|q|=r}|f(q)|$ for $r\in (0, 1)$.
Then the function  $\phi(r)=  \frac{1}{r}(1-r)^{2}M_{\infty}(r,f)$
 is   decreasing on $(0,1)$ and hence tends to a limit  $\alpha\in [0,1]$.
 For any $\alpha\in [0,1]$, there exists $u \in \partial\mathbb{B}$ such that
 $$\lim _{r\nearrow 1^{-}}(1-r)^{2} |f(ru)|=\alpha.$$
\end{theorem}

 \begin{proof}
From  the proof in Theorem \ref{growth-distortion}, we have
\begin{equation}\label{incre-Hayman}
  \frac{ |f(r_{2}u)|}{|f(r_{1}u)| }  \leq \frac{ r_{2}(1-r_{1})^{2}}{r_{1}(1-r_{2})^{2}}.
\end{equation}
Choose $u\in \partial\mathbb{B}$ such  that $M_{\infty}(r_{2},f)= |f(r_{2}u)|$. Whence the above inequality yields that
$$  \frac{1}{r_{2}}(1-r_{2})^{2}M_{\infty}(r_{2},f)  \leq \frac{1}{r_{1}}(1-r_{1})^{2} |f(r_{1}u)| \leq\frac{1}{r_{1}}(1-r_{1})^{2}M_{\infty}(r_{1},f),$$
that is  to say $\phi(r)$ decreases   on $(0,1)$ and hence tends to a limit  $\alpha$. From (\ref{eq:1}), there holds that $ \phi(r)\in [0,1]$, thus $\alpha\in [0,1]$.

Let $\{r_{n}\}$ be a sequence  increasing to $1$ and choose $u_{n}\in \partial\mathbb{B}$ such  that $M_{\infty}(r_{n},f)= |f(r_{n}u_{n})|$.  From   the compactness of $\partial\mathbb{B}$, there exists a cluster point $u_{\infty} \in \partial\mathbb{B}$.
From (\ref{incre-Hayman}), we have, for $r<r_{n}$,
$$ \alpha\leq \frac{1}{ r_{n}}(1-r_{n})^{2} |f(r_{n}u_{n})| \leq \frac{1}{ r}(1-r)^{2} |f(ru_{n})|.$$
Letting $n\rightarrow \infty$, we obtain
$$ \alpha\leq \frac{1}{ r}(1-r)^{2} |f(ru_{\infty})| \leq \frac{1}{ r}(1-r)^{2}M_{\infty}(r ,f),$$
as desired.

\end{proof}

\section{Radius Problems}
In this section,  we  consider  some   radius problems for slice regular functions.
\subsection{Bohr  Theorem}
Based on Lemma \ref{Schwarz-Pick-Cor}, the authors in \cite{Rocchetta-Gentili-Sarfatti}   established   the quaternionic analog of the sharp version of the Bohr theorem saying that

Let $f(q) = \sum _{n=0}^{+\infty} q^{n}a_n $ be a slice regular function on $\mathbb{B}$, continuous on the closure $\mathbb{B}$, such that $|f(q)|< 1$ for all $|q| \leq1$. Then
  \begin{eqnarray} \label{Bohr-theorem}
\sum _{n=0}^{+\infty}|q^{n}a_n|<1, \quad   \ |q| \leq \frac{1 }{3}.
\end{eqnarray}
Moreover $1/3$ is the largest radius for which the statement is true.

 For various  generalizations and variants of the Bohr theorem,
 see e.g., \cite{Muhanna,Bayart,Defant,Hamada} and references therein.

Inequality (\ref{Bohr-theorem}) also can be  written  as
 $$\sum _{n=1}^{+\infty}|q^{n}a_n| < {\rm{dist}}\,(f(0), \partial \mathbb{B}), \quad   \ |q| \leq \frac{1 }{3}.$$

From this view, we can give the  Bohr theorem in a more general setting as follows.
\begin{theorem}\label{Bohr}
Let   $f(q) = \sum _{n=0}^{+\infty} q^{n}a_n $ be a slice regular function on $\mathbb{B}$   such that  $f(q)\in \Pi:=\{q\in \mathbb{H}: {\rm{Re}}\, q  \leq 1 \}$ for all $ q\in \mathbb{B}$. Then
\begin{eqnarray} \label{Bohr-theorem-hyperbolic}
 \sum _{n=1}^{+\infty}|q^{n}a_n| \leq {\rm{dist}}\,(f(0), \partial \Pi), \quad   \ |q| \leq \frac{1 }{3}.
\end{eqnarray}
\end{theorem}

\begin{proof}
Under the condition of Theorem \ref{Bohr}, it follows that (see  \cite[Theorem 4]{Ren})
$$|a_{n}|\leq 2\big(1-{\rm{Re}}\, f(0)\big), \quad n =1, 2, 3, \ldots,$$
which implies that
$$\sum _{n=1}^{+\infty}|q^{n}a_n| \leq 2\big(1-{\rm{Re}}\, f(0)\big)\sum _{n=1}^{+\infty}|q|^{n} \leq 1-{\rm{Re}}\, f(0) = {\rm{dist}}\,(f(0), \partial \Pi), \quad   \ |q| \leq \frac{1 }{3},$$
as desired.
 \end{proof}

\subsection{Rogosinski Theorem}
The well-known theorem of Rogosinski asserts that if the modulus of the holomorphic function on $\mathbb{D}$ is less than $1$, then all  partial sums of its power series are less than $1$ on $\{z\in \mathbb{C}: |z|<\frac{1}{2}\} $ \cite{Rogosinski}.    This theorem has been   generalized    to holomorphic mappings of the open unit ball with values in  an arbitrary convex domain in \cite{Aizenberg} by using the following crucial step.
\begin{lemma}\label{Rogosinski-step}
 Let $g_{m,r}(z) =  \sum _{n=0}^{m}r^{n}z^{n}, z\in  \mathbb{D}, r\in \mathbb{R}$. Then
$${\rm{Re}}\, g_{m,r}(z)\geq \frac{1 }{2}$$
for all $z\in  \mathbb{D}$ and $r \in(0, r_{m})$, where $r_{1} = 1/2, r_{2 }=  \sqrt{3/8}$, and $r_{m} $ for
$m  \geq3$ is the unique positive solution of the equation
$$1-r-2r^{m+1}=0.$$
Moreover, $$1-\frac{\log m}{m}<r_{m}<1-\frac{1}{m},\quad m\geq3.$$
 \end{lemma}

Based on Lemma \ref{Rogosinski-step}, we now establish  the quaternionic analogy of  the Rogosinski theorem as follows.
\begin{theorem}\label{Rogosinski-radii}
Let $f=g+\overline{h}: \mathbb{B}\rightarrow \mathbb{H}$  be   such that the image $f(\mathbb{B})$ is  convex,   where $g(q) = \sum _{n=0}^{+\infty} q^{n}a_n $ and $h(q) = \sum _{n=0}^{+\infty} q^{n}b_n$ are   slice regular functions on $\mathbb{B}$. Then   each  partial sum
$$ S_{m}(q)=\sum _{n=0}^{m} q^{n}a_n+ \overline{\sum _{n=0}^{m} q^{n}b_n } $$
maps $B(0,r_{m})$ into $f(\mathbb{B})$, where $r_{m}$ are given by Lemma \ref{Rogosinski-step}.
 \end{theorem}
\begin{proof}
For any $q\in \mathbb{B}_{I}$, it follows that
$$S_{m}(r_{m}q)=\frac{1}{2\pi}\int_{0}^{2\pi}\big(1+ 2{\rm{Re}}\,\sum _{n=1}^{m}r_{m} ^{n}e^{nI\theta} \big)f(e^{I\theta}q) d\theta.$$

Note that,  for any fixed $I\in \mathbb{S}$ and each positive integer $m$, by Lemma \ref{Rogosinski-step},
$$  \frac{1}{2\pi}\big(1+ 2{\rm{Re}}\,\sum _{n=1}^{m}r_{m} ^{n}e^{nI\theta} \big) d\theta$$
is a probability measure on the interval $[0,2\pi]$ and the  image $f(\mathbb{B})$ is convex, the desired result follows immediately.
\end{proof}

\subsection{One-Quarter  Covering  Theorem}
As an application of  (\ref{eq:1})  in Theorem \ref{growth-distortion},   the Koebe type  one-quarter covering  theorem for the class $\mathcal{S}^{*}$ is obtained.

\begin{theorem}\label{th:Koebe-theorem}
Let $f  \in \mathcal{S}^{*}$,  then it holds that
$$B(0,\frac14)\subset f(\mathbb B).$$
The estimate is precise.
\end{theorem}

\begin{proof}The first inequality in $(\ref{eq:1})$ implies  that
 \begin{eqnarray}\label{Xu-Wa}
 \liminf_{q\rightarrow \partial \mathbb{B}}|f(q)|\geq\frac{1}{4}.
\end{eqnarray}
From the open mapping theorem (see \cite[Theorem 7.7]{GSS}), the image set $f(\mathbb{B})$ is  open containing the origin point  $0$.  This together with (\ref{Xu-Wa}) and \cite[Proposition 3.1]{Xu-Wang} shows that  $f(\mathbb B)$ contains  the ball $B(0,1/4)$.

The preciseness of  estimate can be shown by  the slice regular function
$$f(q)= q(1-q)^{-2} , \quad \forall \ q\in \mathbb{B}.$$
Then
$${\rm{Re}}\,  \big( f(q)^{-1} qf'(q)\big)={\rm{Re}}\, \big((1+q)(1-q)^{-1}\big)=\frac{1-|q|^{2}}{|1-q|^{2}}>0, \quad \forall \ q\in \mathbb{B},$$
which shows that $f\in \mathcal{S}^{*}$.
It is easy to show that  $f(\mathbb B)=\big \{q\in \mathbb{H}: q   \notin (-\infty, -1/4]  \big \}$    contains    $B(0, 1/4)$ while  contains no ball centered at $0$ with radius  larger than $ 1/4$, as desired.
\end{proof}

Now we can offer another    covering theorem from Theorem \ref{starlike-algebra-higher},  instead of Theorem \ref{growth-distortion}.  Specially, taking  $m=2$ in Theorem \ref{starlike-algebra-higher}, we have
\begin{theorem}
Let $ f(q) = q+\sum_{n=3}^{+\infty} q^{n}a_{n}  \in \mathcal{S}^{*}$,  then it holds that
$$B(0,\frac12)\subset f(\mathbb B).$$
\end{theorem}

\subsection{Bloch-Landau Theorem}
In  complex analysis, a well-known result of Bloch-Landau theorem says that if $F$ is a holomorphic function in $\mathbb{D}$ with the only restriction  $F'(0) = 1$, then the image of $F$ contains a disc of radius $r\geq b$, where $ b$ is an absolute constant.

Now we introduce its corresponding version of slice regular functions.
For $a\in \mathbb{B}$ and the slice regular function $f$   on $\mathbb{B}$ with $f'(0)=1$,  denote by $r(a,f)$ the radius of the largest ball contained in the image of $f$   centered at $f(a)$ and denote $r(f)=\sup \{r(a,f), a\in \mathbb{B} \}$.  In \cite{Xu-Wang}, we have proved that $r(f)>0.23$ for any slice regular function $f $ with Bloch seminorm $\sup_{q\in \mathbb{B}}(1-|q|^{2})|f'(q)|=1$.

 In this paper, we restrict  functions $f$ in  $r(f)$ to be   slice regular functions with convex image.  Define
\begin{eqnarray}\label{convex-constant}
 C=\inf \{ r(f) :  f  \ {\rm{is \  slice\ regular \ on\ \mathbb{B}   \ such \ that }}\, \ f'(0)=1 \ {\rm{ and }}\,    f(\mathbb{B}) \ {\rm{ convex}}\,\}.
\end{eqnarray}

Let us prove Theorem \ref{covering-for-convex} before giving the  upper and lower bounds of   $C$.
\begin{proof}[Proof of Theorem \ref{covering-for-convex}]
Let   $f$ be as described in the theorem and assume that  $f(0)=0$ for otherwise we can consider the slice regular function $f-f(0)$.
Let $p\in \partial f(\mathbb{B})$ be a point at minimum distance from  the origin  $0$. From the open mapping theorem for slice regular function \cite{Gentili-Stoppato}, we have $|p|>0$. If $|p|=+\infty$, the theorem has been  proved. Otherwise, $0<|p|<+\infty$, we obtain that
 \begin{eqnarray}\label{convex-condition}
{\rm{Re}}\,  \big( f(q)\overline{p}\big)< |p|^{2}, \quad \forall \ q\in \mathbb{B},
\end{eqnarray}
since $f(\mathbb{B})$ is convex.

Consider the slice regular function
$$ g(q)=(2|p|^{2}-f(q)\overline{p})^{-*}*f(q)\overline{p}, \quad \forall \ q\in \mathbb{B}.$$
Then  it follows that $g(0)=0 $ and $g'(0)=  (2p)^{-1}$.
From  (\ref{convex-condition}), we have
$$|f(q)\overline{p}|<\big|2|p|^{2}-f(q)\overline{p}\big|, \quad \forall \ q\in \mathbb{B},$$
which implies, by (\ref{modulus-inverse}) in Proposition \ref{main-proposition},  $|g(q)|<1$ for all $q\in \mathbb{B}$.
The Schwarz lemma for slice regular functions in Theorem \ref{generalized-Schwarz-lemma} implies that $|g'(0)|\leq1$. Hence,  $|p|\geq 1/2$,
which shows the image $f(\mathbb B)$ contains the open ball $B(0,1/2)$.
To see that the   constant $1/2$ is optimal, we consider the slice regular function given by
$$f(q)=q(1-q)^{-1} , \quad \forall \ q\in \mathbb{B}.$$
 It is easy to show that  $f'(0)=1$  and
 $f(\mathbb B)=\{q\in \mathbb{H}: {\rm{Re}}\, q>-1/2\}$ is convex and contains    the open ball   centered at $0$ of  radius $1/2$ while    $ 1/2$ is optimal, as desired.
\end{proof}

As a consequence of Theorem \ref{covering-for-convex}, we give  a Bloch-Landau theorem for slice regular functions with convex image.

\begin{theorem}The constant $C$  given in $(\ref{convex-constant})$ has the following estimate
  $$\frac{1}{2} \leq C \leq \frac{\pi}{4}. $$
\end{theorem}
\begin{proof}
From Theorem  \ref{covering-for-convex}, it suffices to show that $C \leq \pi/4$.
To this end, let us consider the slice regular function
$$f(q)=\sum_{n=0}^{+\infty} \frac{q^{2n+1}}{2n+1}, \quad q \in \mathbb{B},$$
with $f'(0)=1.$

Note  that the image $f(\mathbb{B})=\{q\in \mathbb{H}: |{\rm{Im}}\, q |<\frac{\pi}{4}\}$ is convex, the desired result follows.
\end{proof}

We have proved the    Bieberbach   conjecture    over quaternions for slice close-to-convex functions. It is still an open problem if the Bieberbach   conjecture holds true for  injective slice regular functions.

$\mathbf{Open \;\;question:}$
Does it hold that  $|a_{n}|\leq n, \ \  n = 2, 3, \ldots,$  provided that  $f(q)= q+ \Sigma_{n=2}^{+\infty}q^{n} a_{n}$ is  an   injective slice regular  function in the open unit ball  $\mathbb{B}$?

\bibliographystyle{amsplain}

\vskip 10mm

\end{document}